\documentclass[11pt]{amsart}
\usepackage{latexsym}
\usepackage{amsmath,amsthm,amssymb, amscd,color}
\usepackage{mathrsfs}
\usepackage[all]{xy}
\usepackage[all]{xy}
\usepackage{tikz}
\usepackage{tikz-cd}
\usepackage{adjustbox}
\usetikzlibrary{snakes}
\usepackage[utf8]{inputenc}
\usepackage{enumerate}
\usepackage{breqn}
\usepackage{hyperref}
\usepackage{listings}
\usepackage{graphicx}
\usepackage{float}
\usepackage{tikz}
\usepackage{xcolor}
\usetikzlibrary{shapes}
\tikzset{
    partial ellipse/.style args={#1:#2:#3}{
        insert path={+ (#1:#3) arc (#1:#2:#3)}
    }
}

\lstdefinestyle{mystyle}{
    language=Python,
    basicstyle=\small\ttfamily,
    showspaces=false,
    showstringspaces=false,
    numbers=left,
    numberstyle=\tiny,
    numbersep=5pt
}

\usepackage{cleveref}

\usepackage{blkarray} 
\usepackage{framed}

\numberwithin{equation}{section}

\theoremstyle{plain}
\newtheorem{theorem}{Theorem}[section]
\newtheorem{lemma}[theorem]{Lemma}
\newtheorem{corollary}[theorem]{Corollary}
\numberwithin{equation}{section}
\newtheorem{proposition}[theorem]{Proposition}

\theoremstyle{definition}
\newtheorem{definition}[theorem]{Definition}
\newtheorem{example}[theorem]{Example}

\newtheorem{remark}[theorem]{Remark}

\theoremstyle{remark}

\newcommand{\RR}{\mathbb{R}}

\begin{document}

\title[Certain classes of polynomial vector fields on the torus]{Characterization and dynamics of certain classes of polynomial vector fields on the torus}

\author[S. Jana]{Supriyo Jana}
\address{Department of Mathematics, Indian Institute of Technology Madras, India}
\email{supriyojanawb@gmail.com}

\subjclass[2020]{34A34, 34C40, 34C45, 34C14}
 \renewcommand{\keywordsname}{Keywords and phrases}

\keywords{Polynomial vector field, First integral, Kolmogorov vector field, pseudo-type-$n$ vector field, Invariant meridian, Invariant parallel, Periodic orbit, Limit cycle, Lie bracket}

\date{\today}
\dedicatory{}

\abstract In this paper, we classify all polynomial vector fields in $\RR^3$ of degree up to three such that their flow makes the torus $$\mathbb{T}^2=\{(x,y,z)\in \RR^3:(x^2+y^2-a^2)^2+z^2-1=0\}~\mbox{with}~a\in (1,\infty)$$ invariant. We also classify cubic Kolmogorov vector fields on $\mathbb{T}^2$ and prove that they exhibit a rational first integral.  We study `pseudo-type-$n$' vector fields on $\mathbb{T}^2$ and show that any such vector field is completely integrable. We prove that the Lie bracket of any two quadratic vector fields on $\mathbb{T}^2$ is completely integrable. We explicitly find all cubic vector fields on $\mathbb{T}^2$ which achieve the sharp bounds for the number of invariant meridians and parallels. We present necessary and sufficient conditions when invariant meridians and parallels of cubic vector fields on $\mathbb{T}^2$ are periodic orbits or limit cycles. We discuss invariant meridians and parallels of pseudo-type-$n$ vector fields as well. Moreover, we characterize the singular points of a class of polynomial vector fields on $\mathbb{T}^2$.   
\endabstract

\maketitle


\section{Introduction}
Polynomial vector fields have been intensely explored since 1900 due to the second part of Hilbert's 16th problem \cite{hilbert1900,ilyashenko2002,MR1334338,li2003hilbert}. People also studied invariant algebraic curves of polynomial vector fields in the plane due to Darboux Integrability Theory \cite{darboux1878memoire}, which states that a sufficient number of invariant algebraic curves guarantees the existence of a first integral. Llibre and Zhang \cite{LlZh02} extended the Darboux Integrability Theory to vector fields on any regular algebraic hypersurface in $\RR^n$. The first integral is also an important part of the qualitative theory of differential systems, see for instance \cite{huang2018meromorphic,llibre2021integrability,llibre2010rational,valls2019liouvillian} and references therein. We remark that a first integral of the associated vector field of a differential system reduces the dimension
of the system by one, making the analysis of the system easier.

We consider the polynomial differential system
\begin{equation}\label{system}
    \frac{dx}{dt}=P,~\frac{dy}{dt}=Q,~\frac{dz}{dt}=R
\end{equation}
 in $\RR^3$ and its associated polynomial vector field 
\begin{equation}\label{vector-field}
     \chi = P\frac{\partial}{\partial x}+Q\frac{\partial}{\partial y}+R\frac{\partial}{\partial z},
\end{equation}
where $P,Q,R\in \mathbb{R}[x,y,z]$. For simplicity, we may write $\chi{:=} (P,Q,R)$. The number $n{:=}\max\{\deg P,~\deg Q,~\deg R\}$ is called the degree of the polynomial vector field \eqref{vector-field}.

Planar Kolmogorov systems were introduced by Kolmogorov \cite{kolmogorov1936sulla} to describe the interaction
between two species occupying the same ecological habitat. Later, polynomial Kolmogorov systems were generalized to higher dimensions with the following form:
\begin{equation*}
    \frac{dx_i}{dt}=x_iP_i(x_1,...,x_n)~\mbox{where}~P_i\in \RR[x_1,...,x_n].
\end{equation*}
These systems may exhibit quite complex dynamical behaviors, see \cite{hirsch1989systems} and the references therein (e.g., references 6, 10, 11, and 12). The associated vector field of a Kolmogorov system is called a Kolmogorov vector field.

The algebraic hypersurface given by $\{f=0\}$ with $f\in \RR[x,y,z]$ is called invariant if $\chi f=Kf$ for some $K \in \mathbb{R}[x,y,z]$. The polynomial $K$ is called the cofactor of $\chi$ for the hypersurface $\{f=0\}$. Moreover, $\chi$ is called a vector field on $\{f=0\}$. In this case, if any solution curve of \eqref{system} intersects the hypersurface $\{f=0\}$ then the entire curve lies in $\{f=0\}$. Polynomial vector fields on the various algebraic hypersurfaces have been studied in \cite{BoLlVa13, gutierrez2002darbouxian, llibre2012rational, LM18, LlMu21, LlRe13, LlZh11}.

Llibre and Medrado \cite{LlMe11} studied polynomial vector fields on the two-dimensional torus $\mathbb{T}^2$ given by 
$$\mathbb{T}^2=\{(x,y,z)\in \RR^3:(x^2+y^2-a^2)^2+z^2-1=0\}~\mbox{where}~a\in (1,\infty).$$ They investigated the sharp upper bounds for the number of invariant meridians, parallels and discussed limit cycles of quadratic vector fields on $\mathbb{T}^2$. The definition of invariant meridians and invariant parallels can be found in Section \ref{sec:prel}. Note that limit cycles are periodic orbits that are isolated in the set of all periodic orbits of the vector field.

In this paper, we classify arbitrary cubic vector fields, cubic Kolmogorov vector fields, and the Lie bracket of two quadratic vector fields on $\mathbb{T}^2$. We discuss the existence of first integrals and characterize the invariant meridians and parallels of these vector fields. In particular, for a cubic vector field on $\mathbb{T}^2$, we provide necessary and sufficient conditions when invariant meridians are limit cycles and invariant parallels are periodic orbits. We prove the following in Section \ref{sec:pol-vfld}.
\begin{proposition}\label{hom-lemma}
    Suppose that $\chi=(P,Q,R)$ is a vector field on $\mathbb{T}^2$ such that $P,Q,R$ are homogeneous polynomials with $\deg P=\deg Q=n$ and $\deg R\leq n$. Then $R=0$.
\end{proposition}
Therefore, we define $\chi=(P,Q,R)$ as a \textit{pseudo-type-$n$ vector field} if $P,Q$ are homogeneous polynomials of degree $n$ and $R=0$. We classify all pseudo-type-$n$ vector fields on $\mathbb{T}^2$ as well.

In Section \ref{sec:inv-mer-par}, we provide a different proof of the sharp bound of invariant meridians for a polynomial vector field on $\mathbb{T}^2$, see \Cref{thm:bound-mer}.
We give a necessary condition for the existence of singular points of a quadratic vector field on $\mathbb{T}^2$. Finally, we discuss the local structure of trajectories on $\mathbb{T}^2$ around the isolated singular points of the vector field $(Ay,-Ax,0)$ where $A\in \RR[x,y,z]$.

\section{Statement of the main results}\label{sec:main-results}
In this section, we state our main results. The following three theorems classify cubic vector fields, the Lie bracket of two quadratic vector fields, and pseudo-type-$n$ vector fields on $\mathbb{T}^2$.
\begin{theorem}\label{thm:cubic-vfld}
    Let $\chi=(P,Q,R)$ be a cubic vector field in $\RR^3$. Then $\chi$ is a vector field on $\mathbb{T}^2$ if and only if there exist $f,K'\in \RR[x,y,z]$ and $\beta,\gamma \in \RR$ such that 
\begin{equation}\label{eq:cubic-vfld}
    \begin{split}
        P=&\frac{1}{4}Kx+fy+\beta z,\\
        Q=&\frac{1}{4}Ky-fx+\gamma z,\\
        R=&\frac{1}{2}K'(-a^2(x^2+y^2)+z^2+a^4-1)-2(\beta x+\gamma y)(x^2+y^2-a^2),
    \end{split}
\end{equation}
where $K=K'z$ and $\deg f,\deg K\leq 2$. Moreover, $K$ is the cofactor of $\chi$ for $\mathbb{T}^2$.
\end{theorem}

\begin{theorem}\label{lie-bracket}
   Let $X$ and $Y$ be two quadratic vector fields on $\mathbb{T}^2$. Then, the Lie bracket $[X,Y]=(Ay,-Ax,0)$ for some $A\in \RR[x,y,z]$ with $\deg A\leq 2$. Moreover, $[X,Y]$ is completely integrable.
\end{theorem}
\begin{theorem}\label{thm:hom-vfld}
Suppose that $\chi=(P,Q,0)$ is a pseudo-type-$n$ vector field in $\RR^3$. Then $\chi$ is a vector field on $\mathbb{T}^2$ if and only if there exists a homogeneous polynomial $A$ of degree $n-1$ such that $P=Ay,~Q=-Ax$. Furthermore, any pseudo-type-$n$ vector field on $\mathbb{T}^2$ is completely integrable.
\end{theorem}
Theorem \ref{thm:cubic-vfld}, Theorem \ref{lie-bracket}, and Theorem \ref{thm:hom-vfld} are proved in Section \ref{sec:pol-vfld}.

We provide a necessary and sufficient condition when a cubic vector field on $\mathbb{T}^2$ having finitely many invariant meridians reaches the sharp bound of the number of invariant meridians, see Theorem \ref{4-mer}. Moreover, we study when those invariant meridians are limit cycles. The next result is proved in Section \ref{sec:inv-mer-par}.
\begin{theorem}\label{mer-periodic}
    Suppose that the cubic vector field \eqref{eq:cubic-vfld} on $\mathbb{T}^2$ has four invariant meridians. Then the invariant meridians are limit cycles if and only if $K'=0$ has no  solution on those meridians. Moreover, the limit cycles appear as stable or unstable alternately.
\end{theorem}
We find all possible invariant meridians and parallels of a cubic Kolmogorov vector field on $\mathbb{T}^2$ having finitely many invariant meridians and parallels. The following result is proved in Section \ref{sec:inv-mer-par}.
\begin{theorem}\label{kolm-mer-par}
    Suppose that a cubic Kolmogorov vector field on $\mathbb{T}^2$ has finitely many invariant meridians and invariant parallels. Then $x=0$ and $y=0$ are the only invariant meridian planes and $z=0$ is the only invariant parallel plane.
\end{theorem}

It is proved in \cite[Theorem 1(a)]{LlMe11} that a degree $m>1$ vector field on $\mathbb{T}^2$ has a maximum of $m-2$ parallel planes. However, Benny and Sarkar \cite[Proposition 6.2]{BeSa22} pointed out that there exists a vector field of degree $m$ on $S^p\times S^1$ having exactly $m-1$ invariant parallel hyperplanes and proved that $m-1$ is the sharp bound.

Hence, a cubic vector field on $\mathbb{T}^2$ has at most two invariant parallel planes. We characterize cubic vector fields on $\mathbb{T}^2$ with exactly two invariant parallel planes. More precisely, we have the following result, which is proved in Section \ref{sec:inv-mer-par}.
\begin{theorem}\label{parallel-periodic}
    If a cubic vector field $\chi=(P,Q,R)$ on $\mathbb{T}^2$ has two invariant parallel planes, then $P,Q,$ and $R$ satisfy
    \begin{equation}\label{1.4}
        \begin{split}
             P=&\frac{1}{2}xz(px+qy)+fy-\frac{a^2p}{2} z,\\
        Q=&\frac{1}{2}yz(px+qy)-fx-\frac{a^2q}{2} z,\\
        R=&(px+qy)(z^2-1),
        \end{split}
    \end{equation}
    for some $f\in \RR[x,y,z]$ with $\deg f\leq 2$ and $p,q\in \RR$.
    
    Moreover, $z-1=0$ and $z+1=0$ are the invariant parallel planes of this vector field. Also, the parallel on  $z-1=0$ is a periodic orbit if and only if $f(x,y,1)-\frac{1}{2}(py-qx)=0$ has no solution on the circle $x^2+y^2=a^2$ and the parallel on $z+1=0$ is a periodic orbit if and only if $f(x,y,-1)+\frac{1}{2}(py-qx)=0$ has no solution on the circle $x^2+y^2=a^2$.
\end{theorem}

In Section \ref{sec:singularity}, we prove the following.
\begin{theorem}\label{thm:quad-singular-pt}
 Suppose that $\chi=(P,Q,R)$ is a quadratic vector field on $\mathbb{T}^2$ such that $R$ is not identically zero. Then $\chi$ does not have any singular point on $\mathbb{T}^2$.
\end{theorem}
We denote $s(\chi)$ as the set of all singular points of the vector field $\chi$.
\begin{proposition}\label{prop:sing-pt}
Suppose that $\chi=(Ay,-Ax,0)$ is a vector field on $\mathbb{T}^2$ for some $A\in \RR[x,y,z]$. Then, $s(\chi)\cap \mathbb{T}^2=\{A=0\}\cap \mathbb{T}^2$. 
\end{proposition}

\section{Preliminaries}\label{sec:prel}

In this section, we present some basic definitions and useful results. We recall the notion of the Lie bracket of two vector fields.

For polynomial vector fields on a regular algebraic hypersurface, \cite[Theorem 3]{llibre2012rational} states that a sufficient number of invariant algebraic hypersurfaces guarantees the existence of a rational first integral. So, it is essential to study different invariant algebraic hypersurfaces for a given vector field. In this paper, we study the invariant planes in $\RR^3$ containing the $z$-axis or orthogonal to the $z$-axis.

We recall the following definitions from \cite{LlMe11}.
The curves obtained by the intersection of a plane containing the $z$-axis with $\mathbb{T}^2$ are called meridians. Similarly, the curves obtained by the intersection of a plane orthogonal to the $z$-axis with $\mathbb{T}^2$ are called parallels. More precisely, meridians can be obtained by $\{ax+by=0\}\cap \mathbb{T}^2$ where $a,b\in \RR$. So, meridians always appear in pairs. The plane $ax+by=0$ is said to be a \textit{meridian plane}. Also, parallels can be obtained by $\{z-k=0\}\cap \mathbb{T}^2$ with $-1\leq k\leq 1$. So, if $k\neq \pm 1$ then parallels appear in pairs. The plane $z-k=0$ is said to be a \textit{parallel plane}.

We say that a meridian obtained by the plane $ax+by=0$ is invariant if $\{ax+by=0\}$ is invariant for $\chi$. Similarly, we say that a parallel obtained by the plane $z-k=0$ is invariant if $\{z-k=0\}$ is invariant for $\chi$. 

We use the tool extactic polynomial to find invariant meridians and parallels. Let
$W$ be a vector subspace of $\RR[x_1,...,x_n]$ generated by the independent
polynomials $v_1 ,... , v_{\ell}$, i.e., $W=\langle v_1,...,v_{\ell} \rangle$. The extactic polynomial of $\chi$ associated with $W$ is the polynomial
$$\mathcal{E}_W(\chi)= 
\begin{vmatrix}
v_1 & \dots &v_{\ell}\\
\chi (v_1) &\dots &\chi (v_{\ell})\\
\cdot & &\cdot \\
\cdot & \cdots&\cdot\\
\cdot & &\cdot \\
\chi^{{\ell}-1}(v_1) &\dots &\chi^{{\ell}-1} (v_{\ell})
\end{vmatrix},
$$
where $\chi^j(v_i) = \underbrace{\chi \circ \chi \circ \cdots \circ \chi}_{j \text{ times}}(v_i)$ for $i\in \{1,\ldots\!, \ell\}$ and $j\in \{2, \ldots\!, {\ell-1}\}$. From the properties of the determinant, it follows that the definition of the extactic polynomial is independent of the chosen
basis of $W$.
The proof of the following is similar to the proof
of \cite[Proposition 1]{LlMe07}.
\begin{proposition}\label{extactic-polynomial}
Let $\chi$ be a polynomial vector field in $\mathbb{R}^n$ and $W$ be 
a finite-dimensional vector subspace of $\RR[x_1, x_2,\ldots\!, x_n]$ with $\dim(W) > 1$. If $\{f=0\}$ is an invariant algebraic hypersurface for the vector field $\chi$ with $f\in W$, then $f$ is a factor of $\mathcal{E}_W(\chi)$.
\end{proposition}
\par The multiplicity of an invariant algebraic hypersurface $f = 0$ with $f\in W$ is the largest positive integer
$k$ such that $f^k$ divides the extactic polynomial $\mathcal{E}_W(\chi)$ when $\mathcal{E}_W(\chi)\neq 0$, otherwise the multiplicity
is infinite. For more details regarding the multiplicity, see \cite{ChLlPe07, LlZh09}. In what follows, for the number of invariant algebraic hypersurfaces, we take into account their multiplicities.
\begin{definition}
Let $U$ be an open subset of $\RR^3$. A non-constant analytic map $H \colon U \to \mathbb{R}$ is called a first integral
of the vector field \eqref{vector-field} on $U$ if $H$ is constant on all solution curves of the system \eqref{system} contained in $U$; i.e., $H(x(t),y(t), z(t)) =$ constant for all values of $t$ for which the solution $(x(t),y(t), z(t))$ is defined and contained in $U$.
\end{definition}
Note that $H$ is
a first integral of the vector field \eqref{vector-field} on $U$ if and only if $\chi H=0$ on $U$. A vector field in $\RR^n$ is called completely integrable if it has $n-1$ functionally independent first integrals. We recall the following two results.
\begin{lemma}\cite[Lemma 3.1]{JaSa23}\label{sum-2-poly-zero}
    Let $Q_1,Q_2,R_1,R_2\in \RR[x_1,...,x_n]$ with $R_1,R_2\neq 0$ be such that  ${Q_1R_1+Q_2R_2}$ is zero and $\gcd(R_1,R_2)=1$. Then $Q_1=AR_2,~Q_2=-AR_1$ for some ${A \in \RR[x_1,...,x_n].}$
\end{lemma}
\begin{lemma}\cite[Lemma 4.1]{JaSa23}\label{sum-3-poly-zero}
    Let $Q_1,Q_2,Q_3\in \RR[x,y,z]$ be such that $Q_1x+Q_2y+Q_3z$ is zero. Then $Q_1=Ay+Bz,~Q_2=-Ax + Cz$, and $Q_3=-Bx-Cy$ for some $A, B, C\in \RR[x,y,z]$.
\end{lemma}
Suppose that $X$ and $Y$ are two vector fields on a smooth manifold $M$. Then, we know that the Lie bracket of $X$ and $Y$ defined by $[X,Y]:=XY-YX$ is also a vector field on $M$. 
More precisely, if  $X=\sum\limits_{i=1}^n P_i \frac{\partial}{\partial x_i}$ and $Y=\sum\limits_{i=1}^n Q_i \frac{\partial}{\partial x_i}$ are two  vector fields, then the Lie bracket of $X$ and $Y$ is defined as $[X,Y]:=\sum\limits_{i=1}^n (X(Q_i)-Y(P_i)) \frac{\partial}{\partial x_i}$. In this paper, we discuss the Lie bracket of polynomial vector fields on $\mathbb{T}^2$ of degree up to two.

Computing the Lie bracket of two vector fields is quite tedious. One can use the following SageMath\footnote{Official website of the open-source software SageMath: \href{https://www.sagemath.org/}{https://www.sagemath.org/}} code to compute the Lie bracket of two vector fields in $\RR^3$. For example, we take two vector fields $P=(P1,P2,P3)$ and $Q=(Q1,Q2,Q3)$ to find their Lie bracket.
\begin{lstlisting}[style=mystyle]
from sage.manifolds.operators import *
var(`') #Define variables
E.<x,y,z> = EuclideanSpace()
def LieBracket(P,Q):
    PQ1 = P.dot(grad(E.scalar_field(Q[1])))
    PQ2 = P.dot(grad(E.scalar_field(Q[2])))
    PQ3 = P.dot(grad(E.scalar_field(Q[3])))
    QP1 = Q.dot(grad(E.scalar_field(P[1])))
    QP2 = Q.dot(grad(E.scalar_field(P[2])))
    QP3 = Q.dot(grad(E.scalar_field(P[3])))
    return E.vector_field(PQ1-QP1,PQ2-QP2,PQ3-QP3)
P=E.vector_field(P1,P2,P3)
Q=E.vector_field(Q1,Q2,Q3)
LieBracket(P,Q).display()
\end{lstlisting}
Throughout this paper, $f^{(j)}$ denotes the degree $j$ homogeneous part of the polynomial $f$.

\section{Polynomial vector fields on $\mathbb{T}^2$}\label{sec:pol-vfld}
In this section, we characterize polynomial vector fields of degrees up to three on $\mathbb{T}^2$. We also classify all cubic Kolmogorov, pseudo-type-$n$ vector fields, the Lie bracket of two quadratic vector fields on $\mathbb{T}^2$ and discuss their Darboux integrability.

\begin{proof}[\textbf{Proof of \Cref{thm:cubic-vfld}}]
    Suppose that $\chi$ is a vector field on $\mathbb{T}^2$. So, $\chi$ satisfies
\begin{equation}\label{eq:cubic-vfld-1}
    4(x^2+y^2-a^2)(Px+Qy)+2Rz=K((x^2+y^2-a^2)^2+z^2-1)
\end{equation}
for some polynomial $K\in \RR[x,y,z]$ with $\deg K\leq 2$. Observe that the constant term of $K$ is zero because the left side of \eqref{eq:cubic-vfld-1} has no constant term.

Thus, $K^{(0)}=0$. Now, comparing the degree 6 terms in \eqref{eq:cubic-vfld-1}, we get
$$4(x^2+y^2)(P^{(3)}x+Q^{(3)}y)=K^{(2)}(x^2+y^2)^2.$$ Hence,
\begin{equation}\label{eq:cubic-vfld-2}
    P^{(3)}x+Q^{(3)}y=\frac{1}{4}K^{(2)}(x^2+y^2).
\end{equation}
We rewrite the above equation as $(P^{(3)}-\frac{1}{4}K^{(2)}x)x+(Q^{(3)}-\frac{1}{4}K^{(2)}y)y=0$. Then, by Lemma \ref{sum-2-poly-zero}, $$P^{(3)}=\frac{1}{4}K^{(2)}x+Ay~\mbox{and}~ Q^{(3)}=\frac{1}{4}K^{(2)}y-Ax$$ for some $A\in \RR[x,y,z]$.

Again, comparing the degree 5 terms in \eqref{eq:cubic-vfld-1}, we get
\begin{equation}\label{eq:cubic-vfld-3}
    4(x^2+y^2)(P^{(2)}x+Q^{(2)}y)=K^{(1)}(x^2+y^2)^2.
\end{equation} Hence, $(P^{(2)}-\frac{1}{4}K^{(1)}x)x+(Q^{(2)}-\frac{1}{4}K^{(1)}y)y=0.$ Then, by Lemma \ref{sum-2-poly-zero}, $$P^{(2)}=\frac{1}{4}K^{(1)}x+By~\mbox{and}~ Q^{(2)}=\frac{1}{4}K^{(1)}y-Bx$$ for some $B\in \RR[x,y,z]$.

The equality in \eqref{eq:cubic-vfld-1} is true for any $x,y,z\in \RR$. In particular, for $z=0$,  \eqref{eq:cubic-vfld-1} becomes
\begin{equation}\label{eq:cubic-vfld-4}
 4(x^2+y^2-a^2)(P(x,y,0)x+Q(x,y,0)y)=K(x,y,0)((x^2+y^2-a^2)^2-1).
\end{equation}
Comparing the degree 4 terms in \eqref{eq:cubic-vfld-4}, we get
\begin{dmath*}
    4(x^2+y^2)(P^{(1)}(x,y,0)x+Q^{(1)}(x,y,0)y)-4a^2(P^{(3)}(x,y,0)x+Q^{(3)}(x,y,0)y)=-2a^2K^{(2)}(x,y,0)(x^2+y^2).
\end{dmath*}
Now, using \eqref{eq:cubic-vfld-2} in the above equation, we obtain
\begin{equation*}
    P^{(1)}(x,y,0)x+Q^{(1)}(x,y,0)y=-\frac{a^2}{4}K^{(2)}(x,y,0).
\end{equation*}
Next, by comparing the degree 2 terms in \eqref{eq:cubic-vfld-4}, we get
\begin{equation*}
    -4a^2(P^{(1)}(x,y,0)x+Q^{(1)}(x,y,0)y)=(a^4-1)K^{(2)}(x,y,0).
\end{equation*}
Hence, from the last two equations, we have $P^{(1)}(x,y,0)x+Q^{(1)}(x,y,0)y=0$ as well as $K^{(2)}(x,y,0)=0$. So, $K^{(2)}=L_1z$ for some $L_1\in \RR[x,y,z]$. Now, comparing the degree 3 terms in \eqref{eq:cubic-vfld-4}, we get
\begin{equation*}
    4(x^2+y^2)(P^{(0)}x+Q^{(0)}y)-4a^2(P^{(2)}(x,y,0)x+Q^{(2)}(x,y,0)y)=-2a^2K^{(1)}(x,y,0)(x^2+y^2).
\end{equation*}
Hence, using \eqref{eq:cubic-vfld-3} in the above equation, we have
\begin{equation*}
    P^{(0)}x+Q^{(0)}y=-\frac{a^2}{4}K^{(1)}(x,y,0).
\end{equation*}
Again, comparing the degree 1 terms in \eqref{eq:cubic-vfld-4}, we have
\begin{equation*}
    -4a^2(P^{(0)}x+Q^{(0)}y)=(a^4-1)K^{(1)}(x,y,0).
\end{equation*}
Hence, from the last two equations, we obtain $K^{(1)}(x,y,0)=0$ as well as $P^{(0)}x+Q^{(0)}y=0$. So, $P^{(0)}=Q^{(0)}=0$ and $K^{(1)}=L_0z$ for some $L_0\in \RR$. Hence, the cofactor $K=K'z$, where $K':=L_1+L_0$.

Now, comparing the degree 1 terms in \eqref{eq:cubic-vfld-1}, we get that $2R^{(0)}z=(a^4-1)L_0z$. This gives $R^{(0)}=\frac{a^4-1}{2}L_0$. Now, comparing the degree 2 terms in \eqref{eq:cubic-vfld-1}, we get
$$-4a^2(P^{(1)}x+Q^{(1)}y)+2R^{(1)}z=(a^4-1)L_1z.$$
Hence, $P^{(1)}x+Q^{(1)}y+(-\frac{1}{2a^2}R^{(1)}+\frac{a^4-1}{4a^2}L_1)z=0.$
Then, by Lemma \ref{sum-3-poly-zero}, $$P^{(1)}=\alpha y+\beta z,~Q^{(1)}=-\alpha x+\gamma z~\mbox{and}~ R^{(1)}=2a^2(\frac{a^4-1}{4a^2}L_1+\beta x+\gamma y)$$ for some $\alpha, \beta, \gamma \in \RR$.
Now, comparing the degree 3 terms in \eqref{eq:cubic-vfld-1}, we get
\begin{equation*}
    -4a^2(P^{(2)}x+Q^{(2)}y)+2R^{(2)}z=L_0z(-2a^2(x^2+y^2)+z^2).
\end{equation*}
Then, using \eqref{eq:cubic-vfld-3} in the above equation, we get
\begin{equation*}
    R^{(2)}=\frac{1}{2}L_0(-a^2(x^2+y^2)+z^2).
\end{equation*}
Next, comparing the degree 4 terms in \eqref{eq:cubic-vfld-1}, we get
\begin{equation}\label{eq:for-r3}
    4(x^2+y^2)(P^{(1)}x+Q^{(1)}y)-4a^2(P^{(3)}x+Q^{(3)}y)+2R^{(3)}z=L_1z(-2a^2(x^2+y^2)+z^2).
\end{equation}
Substituting the values of $P^{(1)}$ and $Q^{(1)}$ in \eqref{eq:for-r3} along with \eqref{eq:cubic-vfld-2}, we obtain
\begin{equation*}
    R^{(3)}=\frac{1}{2}L_1(-a^2(x^2+y^2)+z^2)-2(x^2+y^2)(\beta x+\gamma y).
\end{equation*}
Let us now construct $P,Q,R$.
\begin{equation*}
    \begin{split}
  P&=P^{(3)}+P^{(2)}+P^{(1)}+P^{(0)}\\
  &=\frac{1}{4}K^{(2)}x+Ay+\frac{1}{4}K^{(1)}x+By+\alpha y+\beta z\\
  &=\frac{1}{4}Kx+(A+B+\alpha)y+\beta z.
    \end{split}
\end{equation*}
Assuming $f:=A+B+\alpha$, we get $P=\frac{1}{4}Kx+fy+\beta z$. Also,
\begin{equation*}
    \begin{split}
    Q&=\frac{1}{4}K^{(2)}y-Ax+\frac{1}{4}K^{(1)}y-Bx-\alpha x+\gamma z\\
    &=\frac{1}{4}Ky-fx+\gamma z,\\
    R&=\frac{1}{2}K'(-a^2(x^2+y^2)+z^2+a^4-1)-2(\beta x+\gamma y)(x^2+y^2-a^2).
    \end{split}
\end{equation*}
So, if $\chi=(P,Q,R)$ is a cubic vector field on $\mathbb{T}^2$, then we get $P,Q,$ and $R$ as in \eqref{eq:cubic-vfld}.

Conversely, if $P,Q,$ and $R$ are given by \eqref{eq:cubic-vfld}, then they satisfy \eqref{eq:cubic-vfld-1}. Hence, the converse part also follows.
\end{proof}
\begin{corollary}\label{cor:cubic-kolm}
   Let $\chi=(P,Q,R)$ be a cubic Kolmogorov vector field in $\RR^3$. Then $\chi$ is a vector field on $\mathbb{T}^2$ if and only if there exist $c_1,c_2\in \RR$ such that 
\begin{equation}\label{eq:cubic-kolm}
    \begin{split}
        P=&\frac{1}{4}Kx+c_1xy^2,\\
        Q=&\frac{1}{4}Ky-c_1x^2y,\\
        R=&\frac{c_2}{2}z(-a^2(x^2+y^2)+z^2+a^4-1),
    \end{split}
\end{equation}
where $K=c_2z^2$ is the cofactor of $\chi$ for $\mathbb{T}^2$.
\end{corollary}
\begin{proof}
    By Theorem \ref{thm:cubic-vfld}, $\chi$ is a vector field on $\mathbb{T}^2$ if and only if $P,Q$, and $R$ are given by \eqref{eq:cubic-vfld}. Since $\chi$ is a Kolmogorov vector field, $x,y,$ and $z$ must divide $P,Q,$ and $R$, respectively. If $x$ divides $P$ then $x$ divides $fy+\beta z$. Also, $y$ divides $Q$ implies $y$ divides $-fx+\gamma z$. Suppose that $fy+\beta z=K_1x$ and $-fx+\gamma z=K_2y$ for some $K_1,K_2\in \RR[x,y,z]$. Now, $K_1x^2+K_2y^2=(fy+\beta z)x+(-fx+\gamma z)y=z(\beta x+\gamma y)$. Hence, $\beta=\gamma=0$ and this gives $K_1x^2+K_2y^2=0$. So, $K_1=c_1 y^2$ and $K_2=-c_1 x^2$ for some $c_1 \in \RR$. So,
    $$P=\frac{1}{4}Kx+c_1xy^2~\text{and}~Q=\frac{1}{4}Ky-c_1x^2y.$$ 
    Also, $R=\frac{1}{2}K'(-a^2(x^2+y^2)+z^2+a^4-1)$ since $\beta=\gamma=0$. As $z$ divides $R$, $z$ divides $K'$, say $K'=c_2z$ for some $c_2\in \RR$. So, $R=\frac{c_2}{2}z(-a^2(x^2+y^2)+z^2+a^4-1)$. Hence, we obtain the result.
\end{proof}
\begin{corollary}\label{quad-classification}
If $\chi=(P,Q,R)$ is a quadratic vector field on $\mathbb{T}^2$, then $P,Q,$ and $R$ satisfy
\begin{equation}\label{eq:quad-form}
    \begin{split}
        P=&\frac{\alpha}{4} xz+fy,\\
        Q=&\frac{\alpha}{4} yz-fx,\\
        R=&\frac{\alpha}{2}(-a^2(x^2+y^2)+z^2+a^4-1),
    \end{split}
\end{equation}
where $\alpha\in \RR$ and $f$ is a linear polynomial.    
\end{corollary}
\begin{proof}
    Let $\chi=(P,Q,R)$ is a quadratic vector field on $\mathbb{T}^2$. Then, $P,Q,$ and $R$ are of the form \eqref{eq:cubic-vfld} such that the coefficients of the degree 3 monomials are zero. So, 
    $$\frac{1}{4}K^{(2)}x+f^{(2)}y=0~\mbox{and}~\frac{1}{4}K^{(2)}y-f^{(2)}x=0.$$
    Hence, $\frac{1}{4}K^{(2)}(x^2+y^2)=0$, which implies that $K^{(2)}=0$. Consequently, $f^{(2)}=0$. So, $K=\alpha z$ for some $\alpha\in \RR$. Since $R$ does not contain any degree 3 monomial, we have that $\beta=\gamma=0$. Thus, we get the result.
\end{proof}
Note that Llibre and Medrado \cite{LlMe11} gave a classification of quadratic vector fields on $\mathbb{T}^2$, but the above classification is easy to use.
\begin{corollary}
Any degree one vector field on $\mathbb{T}^2$ is a pseudo-type-1 vector field.
\end{corollary}
\begin{proof}
    Suppose that $(P,Q,R)$ is a degree one vector field on $\mathbb{T}^2$. Hence, $P,Q,$ and $R$ are of the form \eqref{eq:quad-form} without any degree two monomial. Hence, $P=ay,Q=-ax,R=0$ for some $a\in \RR$. Hence, $(P,Q,R)$ is a pseudo-type-1 vector field.
\end{proof}

\begin{proposition}\label{prop:kolm-rational-fi}
    Every cubic Kolmogorov vector field on $\mathbb{T}^2$ has a rational first integral.
\end{proposition}
\begin{proof}
An arbitrary cubic Kolmogorov vector field on $\mathbb{T}^2$ is given by \eqref{eq:cubic-kolm}. Note that $x^2+y^2=0$ and the torus $(x^2+y^2-a^2)^2+z^2-1=0$ are invariant hypersurfaces with the dependent cofactors $\frac{1}{2}K$ and $K$, respectively. Hence, by the Darboux Integrability Theory \cite[Theorem 5]{LlZh02}, $((x^2+y^2-a^2)^2+z^2-1)(x^2+y^2)^{-2}$ is a rational first integral of any cubic Kolmogorov vector field on $\mathbb{T}^2$.
\end{proof}
\begin{remark}
    Using similar arguments of the proof of Proposition \ref{prop:kolm-rational-fi}, the function given by $H:=((x^2+y^2-a^2)^2+z^2-1)(x^2+y^2)^{-2}$ is a rational first integral of any quadratic vector field on $\mathbb{T}^2$ as well.
\end{remark}
\begin{proposition}\label{lie-invariant-hypersurface}
    If the level set $\{f=0\}$ is an invariant algebraic hypersurface of two vector fields $X$ and $Y$ in $\RR^n$, then $\{f=0\}$ is an invariant algebraic hypersurface of the Lie bracket $[X,Y]$ as well.
\end{proposition}
\begin{proof}
The hypersurface $\{f=0\}$ is invariant under the vector fields $X$ and $Y$. So, there exist $K_1,K_2\in \RR[x_1,...,x_n]$ such that $Xf=K_1f$ and $Yf=K_2f$. Now, 
\begin{dmath*}
    [X,Y]f=(XY-YX)(f)=X(Yf)-Y(Xf)=X(K_2f)-Y(K_1f)=X(K_2)f+K_2X(f)-Y(K_1)f-K_1Y(f)=(X(K_2)-Y(K_1))f.
\end{dmath*}
Hence, $\{f=0\}$ is invariant for the vector field $[X,Y]$ with cofactor $X(K_2)-Y(K_1)$.
\end{proof}
\begin{remark}
    If two vector fields have a common first integral $f$, then their Lie bracket also has the first integral $f$.
\end{remark}
\begin{proof}[\textbf{Proof of \Cref{lie-bracket}}]
Let $X=(P_0,Q_0,R_0)$ and $Y=(P_1,Q_1,R_1)$, where
$$P_0=\frac{\alpha_0}{4}xz+f_0y,~Q_0=\frac{\alpha_0}{4}yz-f_0x,~R_0=\frac{\alpha_0}{2}(-a^2(x^2+y^2)+z^2+a^4-1)$$
and
$$P_1=\frac{\alpha_1}{4} xz+f_1y,~Q_1=\frac{\alpha_1}{4}yz-f_1x,~R_1=\frac{\alpha_1}{2}(-a^2(x^2+y^2)+z^2+a^4-1)$$
for some $\alpha_0,\alpha_1\in \RR$ and linear polynomials $f_0,f_1$. Now, we compute the Lie bracket $[X,Y]$. If $[X,Y]=(P,Q,R)$ then one can show that $\deg P,\deg Q\leq 3$ and $R=0$. By Proposition \ref{lie-invariant-hypersurface}, $[X,Y]$ is a vector field on $\mathbb{T}^2$. So, we have the following.
\begin{equation}\label{lie-bracket-vfld}
    4(x^2+y^2-a^2)(Px+Qy)=K((x^2+y^2-a^2)^2+z^2-1)
\end{equation} for some $K\in \RR[x,y,z]$ with $\deg K\leq 2$.

Hence, $x^2+y^2-a^2$ divides $K$, say $K=c(x^2+y^2-a^2)$ for some $c\in \RR$. Then \eqref{lie-bracket-vfld} becomes
\begin{equation*}
    4(Px+Qy)=c((x^2+y^2-a^2)^2+z^2-1).
\end{equation*} Since the left side has no term $z^2$, we obtain $c=0$. This gives $Px+Qy=0$. Then, by Lemma \ref{sum-2-poly-zero}, $P=Ay,~Q=-Ax$ for some $A\in \RR[x,y,z]$ with $\deg A\leq 2$.

The functions $f:=x^2+y^2$ and $g:=z$ are two functionally independent first integrals of $[X,Y]$, which makes $[X,Y]$ completely integrable.
\end{proof}

\begin{proof}[\textbf{Proof of \Cref{hom-lemma}}]
    The following must be satisfied since $\chi$ is a vector field on $\mathbb{T}^2$.
    \begin{equation}\label{eq:torus-inv}
        4(x^2+y^2-a^2)(Px+Qy)+2Rz=K((x^2+y^2-a^2)^2+z^2-1)
    \end{equation}
    for some polynomial $K$ with $\deg K\leq n-1$.
    
    Suppose that $Px+Qy\neq 0$. Then $\deg K=n-1$. Observe that the degree of each monomial on the left side of \eqref{eq:torus-inv} is either $n+1,~n+3$, or $\deg R+1$. Hence, $\deg R=n-2$, and $K$ is a homogeneous polynomial. Note that $2Rz=K(a^4-1)$. So, by rewriting \eqref{eq:torus-inv}, we get
    $$4(x^2+y^2-a^2)(Px+Qy)=K((x^2+y^2-a^2)^2+z^2-a^4).$$
    So, $x^2+y^2-a^2$ must divide $K(z^2-a^4)$, which is not possible. So, $Px+Qy$ must be zero. Then \eqref{eq:torus-inv} implies that
    $$2Rz=K((x^2+y^2-a^2)^2+z^2-1).$$
    The above equation is valid only when $K=0$ because $R$ is a homogeneous polynomial. Hence, the result is proved.
\end{proof}

\begin{proof}[\textbf{Proof of \Cref{thm:hom-vfld}}]
Proof of the first part is similar to the proof of \cite[Theorem 4.11]{JaSa23}. To prove the second part, notice that $f:=x^2+y^2$ and $g:=z$ are two functionally independent first integrals for $\chi$.
\end{proof}

\begin{corollary}\label{lie-pseudo-2}
     The Lie bracket of two pseudo-type-2 vector fields on $\mathbb{T}^2$ is either zero or a pseudo-type-3 vector field.
\end{corollary}
\begin{proof}
Let $X$ and $Y$ be two pseudo-type-2 vector fields on $\mathbb{T}^2$. Then, by Theorem \ref{thm:hom-vfld}, $X=(Ay,-Ax,0)$ and $Y=(By,-Bx,0)$ for some homogeneous degree one polynomials $A$ and $B$. By Theorem \ref{lie-bracket}, $[X,Y]=(Cy,-Cx,0)$ for some $C\in \RR[x,y,z]$ with $\deg C\leq 2$. Observe that $C$ is either zero or a quadratic homogeneous polynomial since $A$ and $B$ are homogeneous. Hence, the Lie bracket of two pseudo-type-2 vector fields on $\mathbb{T}^2$ is either zero or a pseudo-type-3 vector field.
\end{proof}
Note that any degree one vector field on $\mathbb{T}^2$ is a pseudo-type-1 vector field. Hence, the Lie bracket of any two degree one  vector fields on $\mathbb{T}^2$ is zero. So, the Lie brackets of two degree one vector fields and two degree two vector fields on $\mathbb{T}^2$ have the $z$-component always zero. But this is not true in general.
\begin{example}
    Consider the following two vector fields in $\RR^3$:
    $$X=(x^2z,xyz,2x(-a^2(x^2+y^2)+z^2+a^4-1))~\mbox{and}~Y=(y^3,-xy^2,0).$$
    $X$ and $Y$ are vector fields on $\mathbb{T}^2$ due to Theorem \ref{thm:cubic-vfld} with the parameter sets\newline
    \{$K=4xz,f=\beta=\gamma=0$\} and \{$f=y^2,K=\beta=\gamma=0$\}, respectively. Then, the $z$-component of $[X,Y]$ is $-2y^3(-a^2(x^2+y^2) +z^2 +a^4 -1)$.
\end{example}

\section{Invariant meridians and parallels}\label{sec:inv-mer-par}
In this section, we study invariant meridians and parallels of cubic vector fields on $\mathbb{T}^2$. We find all possible invariant meridians and parallels of a cubic Kolmogorov vector field on $\mathbb{T}^2$. We show necessary and sufficient conditions for a cubic vector field on $\mathbb{T}^2$ when invariant meridians and parallels are periodic orbits or limit cycles. We find the maximum number of invariant meridians and parallels of pseudo-type-$n$ vector fields on $\mathbb{T}^2$. We also provide a different and simple proof for the sharp bound of invariant meridians of a polynomial vector field on $\mathbb{T}^2$.

Llibre and Medrado \cite[Theorem 1]{LlMe11} proved that if a vector field $\chi$ of degree $m>1$ on $\mathbb{T}^2$ has finitely many invariant meridians then it has at most $2(m-1)$ invariant meridians. While the statement is correct, we have identified an oversight in the proof. The authors computed that $\mathcal{E}_{\langle x,y\rangle}(\chi)=(x^2+y^2)\dot{\theta}$ and hence followed $x^2+y^2$ divides $\mathcal{E}_{\langle x,y\rangle}(\chi)$. But this is true only when $\dot{\theta}$ is a polynomial. For example, consider the vector field $\chi$ on $\mathbb{T}^2$ of the form \eqref{eq:cubic-vfld} with $\beta,\gamma\neq 0$. Then, $x^2+y^2$ is not a factor of $\mathcal{E}_{\langle x,y\rangle}(\chi)$. Here, we provide a different proof.
\begin{theorem}\label{thm:bound-mer}
    Suppose that a degree $n$ polynomial vector field $\chi$ on $\mathbb{T}^2$ has finitely many invariant meridians taking into account multiplicities. Then, $\chi$ can have at most $2(n-1)$ invariant meridians. Moreover, this bound is sharp.
\end{theorem}
\begin{proof}
   Suppose that $\chi=(P,Q,R)$. Since $\chi$ is a vector field on $\mathbb{T}^2$, it satisfies
\begin{equation}\label{eq:n-vfld}
    4(x^2+y^2-a^2)(Px+Qy)+2Rz=K((x^2+y^2-a^2)^2+z^2-1)
\end{equation}
for some $K\in \RR[x,y,z]$ with $\deg K\leq n-1$. Now, by comparing the degree $n+3$ and $n+2$ terms in \eqref{eq:n-vfld}, we get
\begin{equation}\label{eq:n+3-terms}
    4(x^2+y^2)(P^{(n)}x+Q^{(n)}y)=K^{(n-1)}(x^2+y^2)^2
\end{equation}
and 
\begin{equation}\label{eq:n+2-terms}
    4(x^2+y^2)(P^{(n-1)}x+Q^{(n-1)}y)=K^{(n-2)}(x^2+y^2)^2,
\end{equation}
respectively. Hence,
$$P^{(n)}x+Q^{(n)}y=\frac{x^2+y^2}{4}K^{(n-1)}~\mbox{and}~P^{(n-1)}x+Q^{(n-1)}y=\frac{x^2+y^2}{4}K^{(n-2)}.$$
So, by Lemma \ref{sum-2-poly-zero}, $$P^{(n)}=\frac{x}{4}K^{(n-1)}+Ay,~Q^{(n)}=\frac{y}{4}K^{(n-1)}-Ax,$$ 
$$P^{(n-1)}=\frac{x}{4}K^{(n-2)}+By, ~Q^{(n-1)}=\frac{y}{4}K^{(n-2)}-Bx,$$ for some $A,B\in \RR[x,y,z]$ with $\deg A= n-1, \deg B= n-2$.
The extactic polynomial of $\chi$ associated with $\langle x,y\rangle$ is 
$\mathcal{E}_{\langle x,y\rangle}(\chi)=\begin{vmatrix}
    x&y\\
    P&Q
\end{vmatrix}=Qx-Py$. Observe that the degree of $\mathcal{E}_{\langle x,y\rangle}(\chi)$ is at most $n+1$. Suppose that $\mathcal{E}_{\langle x,y\rangle}(\chi)$ has $\ell$ factors of the form $ax+by$, where $a,b\in \RR$. Due to Proposition \ref{extactic-polynomial}, it is enough to prove that $\ell \leq n-1$.

Let $\ell>n-1$. Then, $\mathcal{E}_{\langle x,y\rangle}(\chi)=(Q^{(n)}+Q^{(n-1)})x-(P^{(n)}+P^{(n-1)})y=-(A+B)(x^2+y^2)$. Note that $x^2+y^2$ has no linear factor over $\RR$. So, we get a contradiction. Thus, $\ell \leq n-1$.

To prove the bound is sharp, one can see the example in the proof of \cite[Theorem 1]{LlMe11}.
\end{proof}
\begin{proposition}\label{ch-invariant-mer}
A pseudo-type-$n$ vector field on $\mathbb{T}^2$ can have at most $2(n-1)$ invariant meridians and this bound is sharp. Any invariant meridian of a pseudo-type-$n$ vector field $\chi$ on $\mathbb{T}^2$ is a subset of $s(\chi)$.
\end{proposition}
\begin{proof}
Consider an arbitrary pseudo-type-$n$ vector field  $(Ay,-Ax,0)$ on $\mathbb{T}^2$ for some $A\in \RR[x,y,z]$ of degree $n-1$. Then, the extactic polynomial associated with $\langle x,y\rangle$ is $$\mathcal{E}_{\langle x,y\rangle}(\chi)=\begin{vmatrix}
        x&y\\
        Ay&-Ax
    \end{vmatrix}=-A(x^2+y^2).$$
If $ax+by=0$ gives an invariant meridian, then by Proposition \ref{extactic-polynomial}, $ax+by$ must divide $A$. Observe that $A$ can have at most $n-1$ factors of the form $ax+by$. Hence, $(Ay,-Ax,0)$ can have at most $2(n-1)$ invariant meridians. To show the bound is sharp, consider $A=\prod_{i=1}^{n-1}(a_ix+b_iy)$. Then, $a_ix+b_iy=0$ are invariant meridian planes of $\chi$ for $i=1,\ldots\!, n-1$. So, the vector field has $2(n-1)$ invariant meridians.

From the first part of the proof, we observe that if $ax+by=0$ is an invariant meridian plane then $ax+by$ is a factor of $A$. Hence, any point on the invariant meridian is a singular point of $\chi$.
\end{proof}
\begin{remark}
    No invariant meridian of a pseudo-type-$n$ vector field on $\mathbb{T}^2$ is a periodic orbit.
\end{remark}
Next, we look at the cubic vector fields on $\mathbb{T}^2$ with 2(3-1)=4 invariant meridians.
\begin{theorem}\label{4-mer}
    Suppose that $\chi$ is a cubic vector field on $\mathbb{T}^2$ given by \eqref{eq:cubic-vfld}. Then $\chi$ has exactly four invariant meridians given by the planes $a_ix+b_iy=0$ for $i=1,2$ if and only if $f=c\prod\limits_{i=1}^2 (a_ix+b_iy)$ for some $c\in \RR\setminus \{0\}$ and $\beta=\gamma=0$.
\end{theorem}
\begin{proof}
Suppose that $\chi$ has exactly four invariant meridians given by the planes $a_ix+b_iy=0$ for $i=1,2$. The extactic polynomial of $\chi=(P,Q,R)$ given by \eqref{eq:cubic-vfld} associated with $W=\langle x,y\rangle$ is
$$ \mathcal{E}_W(\chi)=\begin{vmatrix}
        x &y\\
        P &Q
    \end{vmatrix}=\begin{vmatrix}
        x &y\\
        \frac{1}{4}Kx+fy+\beta z &\frac{1}{4}Ky-fx+\gamma z
    \end{vmatrix}=-f(x^2+y^2)+z(\gamma x -\beta y).$$
Notice that $\mathcal{E}_W(\chi)$ is a degree 4 polynomial. By Proposition \ref{extactic-polynomial}, there exists $g\in \RR[x,y,z]$ of degree 2 such that
\begin{equation}\label{mer-extactic}
    -f(x^2+y^2)+z(\gamma x -\beta y)=\mathcal{E}_W(\chi)=g\prod\limits_{i=1}^2(a_ix+b_iy).
\end{equation}
Hence, $-f^{(2)}(x^2+y^2)=g^{(2)}\prod\limits_{i=1}^2(a_ix+b_iy)$. We know that $x^2+y^2$ has no linear factor over $\RR$. Thus, $\prod\limits_{i=1}^2(a_ix+b_iy)$ must divide $f^{(2)}$. So, $f^{(2)}=c \prod\limits_{i=1}^2(a_ix+b_iy)$ for some $c\in \RR$. Similarly, $\prod\limits_{i=1}^2(a_ix+b_iy)$ divides $f^{(1)}$ which implies that $f^{(1)}=0$. Now, by comparing the degree 2 terms in \eqref{mer-extactic}, we obtain
$$-f^{(0)}(x^2+y^2)+z(\gamma x-\beta y)=g^{(0)}\prod\limits_{i=1}^2(a_ix+b_iy).$$
This gives $\beta=\gamma=0$ since the coefficients of $xz$ and $yz$ on the right side are zero. Hence, $-f^{(0)}(x^2+y^2)=g^{(0)}\prod\limits_{i=1}^2(a_ix+b_iy)$, which is possible only when $f^{(0)}=g^{(0)}=0$. Therefore, $f=c \prod\limits_{i=1}^2(a_ix+b_iy)$. Observe that $c$ must be non-zero since the extactic polynomial $\mathcal{E}_W(\chi)$ is non-zero. One can check that the converse part also holds.
\end{proof}

\begin{proof}[\textbf{Proof of \Cref{mer-periodic}}]
Consider the cubic vector field $\chi$ on $\mathbb{T}^2$ given by \eqref{eq:cubic-vfld}. Suppose that $\chi$ has four invariant meridians given by the planes $a_ix+b_iy=0$ for $i=1,2$. Then, by Theorem \ref{4-mer}, we get $f=c\prod\limits_{i=1}^2(a_ix+b_iy)$ for some $c\in \RR\setminus \{0\}$ and $\beta=\gamma=0$. In cylindrical coordinates, the vector field $\chi$ becomes $(\dot{r},\dot{\theta},\dot{z})$, where
\begin{equation*}
    \begin{split}
        \dot{r}&=\frac{1}{4}rzK'(r\cos\theta,r\sin\theta,z),\\
        \dot{\theta}&=-f=-c\prod\limits_{i=1}^2(a_ir\cos \theta+b_ir\sin \theta),\\
        \dot{z}&=\frac{1}{2}K'(-a^2r^2+z^2+a^4-1).
    \end{split}
\end{equation*}
Suppose that $(r_0,\theta_0,z_0)\in \mathbb{T}^2$ is a singular point of $\chi$. Observe that $r_0\neq 0$. Hence, from $\dot{r}=0$, we obtain $zK'(r_0\cos \theta_0,r_0\sin \theta_0,z_0)=0$. If $K'(r_0\cos \theta_0,r_0\sin \theta_0,z_0)\neq 0$ then $z_0=0$. In this case, $(r_0^2-a^2)^2=1$ since $(r_0,\theta_0,z_0)\in \mathbb{T}^2$. Also, from $\dot{z}=0$, we get $-a^2r_0^2+a^4-1=0$ which gives $(r_0^2-a^2)^2=\frac{1}{a^4}$. Thus, $1=\frac{1}{a^4}$ which contradicts the fact that $a>1$. Hence, $K'(r_0\cos \theta_0,r_0\sin \theta_0,z_0)$ must be zero. Also, notice that $\dot{\theta}=0$ on the invariant meridians. So, a point on the invariant meridians is a singular point if and only if $K'=0$ at that point. It is well known that an invariant closed curve is a periodic orbit if and only if the curve does not contain any singular point. Hence, the invariant meridians are periodic orbits if and only if $K'=0$ has no solution on the meridians.

Suppose that the invariant meridians are periodic orbits. If $\Gamma$ is another periodic orbit, except for the invariant meridians, then it must lie between two consecutive invariant meridians. Hence, the $\theta$-component of $\Gamma$ must take a value twice, depending on the time. Therefore, by Rolle's theorem, $\dot{\theta}=0$ at some point on $\Gamma$. We observe that $\dot{\theta}=0$ only on the invariant meridians. So, there are no periodic orbits on $\mathbb{T}^2$ except the invariant meridians. Hence, the invariant meridians are limit cycles. Between two consecutive invariant meridians $\dot{\theta}$ is either positive or negative. Also, the sign of $\dot{\theta}$ changes only when we cross one invariant meridian because $$\dot{\theta}=-cr^2r_1r_2 \cos (\theta-\theta_1)\cos(\theta-\theta_2),$$ where $\theta_1,\theta_2\in \RR$ and $r_1,r_2>0$. So, the invariant meridians are stable or unstable limit cycles alternately.
\end{proof}
\begin{example}
    Consider the cubic vector field $\chi=(P,Q,R)$ with
    \begin{equation*}
        \begin{split}
            P&=\frac{1}{4}xz+xy^2,\\
            Q&=\frac{1}{4}yz-x^2y,\\
            R&=\frac{1}{2}(-a^2(x^2+y^2)+z^2+a^4-1).
        \end{split}
    \end{equation*}
By Theorem \ref{thm:cubic-vfld}, $\chi$ is a vector field on $\mathbb{T}^2$. Observe that $\chi$ has exactly four invariant meridians generated by the planes $x=0$ and $y=0$. Hence, by Theorem \ref{mer-periodic}, the invariant meridians are stable or unstable limit cycles alternately. Notice that  $\dot{\theta}=-\frac{r^2}{2}\sin 2\theta$.

\begin{figure}[H]
    \centering
    \begin{tikzpicture}
\draw(0,0) ellipse (2.3 and 1.5); 
\draw(0,0) ellipse (1.5 and 0.7); 

\draw (0,1.5) arc(90:270:0.2cm and 0.4cm); 
\draw[style=densely dashed] (0,1.5) arc(90:-90:0.2cm and 0.4cm); 

\draw (0,-0.7) arc(90:270:0.2cm and 0.4cm); 
\draw[style=densely dashed] (0,-0.7) arc(90:-90:0.2cm and 0.4cm); 

\draw (-1.5,0) arc(0:180:0.4cm and 0.2cm); 
\draw[style=densely dashed] (-1.5,0) arc(0:-180:0.4cm and 0.2cm); 

\draw (2.3,0) arc(0:180:0.4cm and 0.2cm); 
\draw[style=densely dashed] (2.3,0) arc(0:-180:0.4cm and 0.2cm); 

\draw[red, <-] (0,0) [partial ellipse=13:93:1.8cm and 1cm];
\draw[red, <-] (0,0) [partial ellipse=11:93:2cm and 1.25cm];

\draw[red, <-] (0,0) [partial ellipse=167:97:1.8cm and 1cm];
\draw[red, <-] (0,0) [partial ellipse=168:98:2cm and 1.25cm];

\draw[red, <-] (0,0) [partial ellipse=5:-93:1.8cm and 1cm];
\draw[red, <-] (0,0) [partial ellipse=5:-93:2cm and 1.25cm];

\draw[red, <-] (0,0) [partial ellipse=-184:-97:1.8cm and 1cm];
\draw[red, <-] (0,0) [partial ellipse=-184:-97:2cm and 1.25cm];
\end{tikzpicture}
    \caption{The left and right meridians are stable limit cycles, whereas the top and bottom meridians are unstable limit cycles.}
    \label{fig:fig1}
\end{figure}
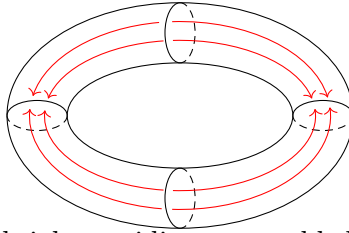
\end{example}
\begin{proof}[\textbf{Proof of \Cref{kolm-mer-par}}]
Consider the cubic Kolmogorov vector field $\chi=(P,Q,R)$ on $\mathbb{T}^2$ given by \eqref{eq:cubic-kolm}. Then the extactic polynomial of $\chi$ associated with $\langle x,y\rangle$ is
    $$\mathcal{E}_{\langle x,y\rangle}(\chi)=\begin{vmatrix}
        x&y\\
        \frac{1}{4}Kx+c_1xy^2& \frac{1}{4}Ky-c_1x^2y
    \end{vmatrix}=-c_1xy(x^2+y^2).$$ The only divisors of the form $ax+by$ of the above extactic polynomial are $x$ and $y$. Observe that $x=0$ and $y=0$ are invariant planes of $\chi$. By Proposition \ref{extactic-polynomial}, these are the only invariant meridian planes of $\chi$.

    Note that $z-k=0$ is an invariant plane of $\chi$ if and only if $z-k$ divides $R$. But the only divisor of $R$ of the form is $z-k$ is $z$. Hence, $z=0$ is the only invariant parallel plane of $\chi$.
\end{proof}
\begin{remark}
    For a cubic Kolmogorov vector field on $\mathbb{T}^2$, no invariant meridian or parallel is a periodic orbit since each one contains some singular points (e.g. $(0,\pm \sqrt{a^2\pm 1},0)$ or $(\pm \sqrt{a^2\pm 1},0,0)$).
\end{remark}
\begin{proposition}\label{pseudo-2-mer}
A non-zero Lie bracket of two pseudo-type-2 vector fields on $\mathbb{T}^2$ can have a maximum of two invariant meridians.
\end{proposition}
\begin{proof}
    Suppose that $X=(Ay,-Ax,0)$ and $Y=(By,-Bx,0)$ are two pseudo-type-2 vector fields on $\mathbb{T}^2$ with $A=a_1x+a_2y+a_3z$ and $B=b_1x+b_2y+b_3z$. Then $[X,Y]=(Cy,-Cx,0)$ where $C=(a_2b_1-a_1b_2)(x^2+y^2)+(a_2b_3-a_3b_2)xz+(a_3b_1-a_1b_3)yz$. Hence, the extactic polynomial  $\mathcal{E}_{\langle x,y\rangle}([X,Y])=-C(x^2+y^2)$. One can check that $C$ can have at most one factor of the form $ax+by$. Hence, the result follows.
\end{proof}

Note that any parallel plane $\{z-k=0\}$ is invariant for a pseudo-type-$n$ vector field on $\mathbb{T}^2$.
\begin{proof}[\textbf{Proof of \Cref{parallel-periodic}}]
Suppose that $\chi$ is given by \eqref{eq:cubic-vfld} and $z-k_1=0$, $z-k_2=0$ are two invariant parallel planes. Then, $(z-k_1)(z-k_2)$ divides $R$ where $$R=\frac{1}{2}K'(-a^2(x^2+y^2)+z^2+a^4-1)-2(\beta x+\gamma y)(x^2+y^2-a^2).$$ Assume that $\frac{1}{2}K'=px+qy+rz+s$. So,
\begin{equation}\label{parallel-div}
   (px+qy+rz+s)(-a^2(x^2+y^2)+z^2+a^4-1)-2(\beta x+\gamma y)(x^2+y^2-a^2)=L \prod\limits_{i=1}^2(z-k_i)
\end{equation}
for some linear polynomial $L\in \RR[x,y,z]$. We compare the coefficients of $x^3,y^3$ in \eqref{parallel-div} and we get $-a^2p-2\beta=0$ and $-a^2q-2\gamma=0$, respectively. This gives $\beta=-\frac{a^2p}{2}$ and $\gamma=-\frac{a^2q}{2}$. Also, by comparing the coefficients of $x^2z,x^2$ in \eqref{parallel-div}, we get $r=s=0$. Hence, \eqref{parallel-div} becomes
\begin{equation*}\label{parallel-div-1}
     (px+qy)(z^2-1)=L\prod\limits_{i=1}^2(z-k_i).
\end{equation*}
This implies that $L=px+qy,k_1=1,$ and $k_2=-1$. Also, $K=K'z=2(px+qy)z$. Hence, we get $P,Q,$ and $R$ as mentioned in \eqref{1.4}.

One can check that $z-1=0$ and $z+1=0$ are the only invariant parallel planes of $\chi$. Now, we obtain the necessary and sufficient condition for the parallel on  $z-1=0$ to be a periodic orbit. The discussion for the parallel on $z+1=0$ is similar. In cylindrical coordinates the vector field $\chi$ becomes $(\dot{r},\dot{\theta},\dot{z})$ where
\begin{equation*}
    \begin{split}
        \dot{r}&=\frac{r^2-a^2}{2}(p\cos \theta+q\sin \theta)z,\\
        \dot{\theta}&=-f+\frac{a^2}{2r^2}(pr\sin \theta-qr\cos \theta)z,\\
        \dot{z}&=r(p\cos \theta+q\sin \theta)(z^2-1).
    \end{split}
\end{equation*}
Since $z-1=0$ is invariant, thus the parallel generated by this plane is a periodic orbit if and only if it does not contain any singular point. Suppose that $(r_0,\theta_0,z_0)\in \mathbb{T}^2$ lies on $z-1=0$. Hence, $z_0=1$ and $r_0=a$. At $(a,\theta_0,1)$, $\dot{r}=\dot{z}=0$. Observe that $\dot{\theta}=0$ at $(a,\theta_0,1)$ if and only if $f(a\cos \theta_0,a\sin \theta_0,1)=\frac{1}{2}(pa\sin \theta_0-qa\cos \theta_0)$. Hence, the result follows. 
\end{proof}
\begin{theorem}\label{infinite-par}
    Suppose that $\chi=(P,Q,R)$ is a cubic vector field on $\mathbb{T}^2$. Then $\chi$ has infinitely many invariant parallels if and only if $P=fx,~Q=-fy,~R=0$ for some quadratic polynomial $f\in \RR[x,y,z]$.
\end{theorem}
\begin{proof}
    Suppose that $\chi$ is given by \eqref{eq:cubic-vfld} and it has infinitely many invariant parallels. Then, we must have
    $$R=\frac{1}{2}K'(-a^2(x^2+y^2)+z^2+a^4-1)-2(\beta x+\gamma y)(x^2+y^2-a^2)=0.$$ Thus,
$$\frac{1}{2}K'(z^2-1)=(x^2+y^2-a^2)(\frac{a^2}{2}K'+2(\beta x+\gamma y)).$$
    So, $x^2+y^2-a^2$ divides $K'$ and this is possible only when $K'=0$ since $\deg K'\leq 1$. Consequently, $K=0$ and $\beta=\gamma=0$. So, we get $P=fy,~Q=-fx$.
    
One can check that the converse part is also true since $R=0$.
\end{proof}

\section{Singularities of polynomial vector fields on $\mathbb{T}^2$}\label{sec:singularity}
In this section, we show the non-existence of singular points of some polynomial vector fields on $\mathbb{T}^2$. We also classify the singular points of a large class of polynomial vector fields on $\mathbb{T}^2$.

\begin{proof}[\textbf{Proof of \Cref{thm:quad-singular-pt}}]
    Consider a quadratic vector field $\chi=(P,Q,R)$ on $\mathbb{T}^2$. Then, by \Cref{quad-classification},
$$P=\frac{\alpha}{4} xz+fy,~Q=\frac{\alpha}{4} yz-fx,~R=\frac{\alpha}{2}(-a^2(x^2+y^2)+z^2+a^4-1)$$
for some $\alpha\in \RR$ and a linear polynomial $f$. If $\alpha=0$ then $R$ is identically zero. So, we assume that $\alpha \neq 0$. Let $(x_0,y_0,z_0)\in \mathbb{T}^2$ is a singular point of $\chi$. Then $(x_0^2+y_0^2-a^2)^2+z_0^2=1$. Moreover, $R(x_0,y_0,z_0)=0$ gives $$
(x_0^2+y_0^2)(x_0^2+y_0^2-a^2)=0.$$
If $x_0^2+y_0^2=0$, then $z_0^2=1-a^4$. But this is not possible since $a>1$. Hence, $x_0^2+y_0^2=a^2$ and $z_0^2=1$. Note that $P(x_0,y_0,z_0)y_0-Q(x_0,y_0,z_0)x_0=0$ which implies that $(x_0^2+y_0^2)f(x_0,y_0,z_0)=0$. So, $f(x_0,y_0,z_0)=0$. As $P(x_0,y_0,z_0)=Q(x_0,y_0,z_0)=0$, thus $x_0=y_0=0$, which contradicts the fact that $x_0^2+y_0^2=a^2$. Therefore, $\chi$ does not have any singular point on $\mathbb{T}^2$.
\end{proof}

\begin{proof}[\textbf{Proof of \Cref{prop:sing-pt}}]
Suppose that $(x_0,y_0,z_0)\in \mathbb{T}^2$ is a singular point of $\chi$. Then, $A(x_0,y_0,z_0)y_0=0$ and $A(x_0,y_0,z_0)x_0=0$. Hence, $A(x_0,y_0,z_0)=0$ since $x_0$ and $y_0$ cannot be simultaneously zero. Hence, the statement is proved.
\end{proof}
\begin{remark}
    A degree one vector field on $\mathbb{T}^2$ has no singular point on $\mathbb{T}^2$.
\end{remark}

We now study the isolated singularities of the vector field $(Ay,-Ax,0)$ on $\mathbb{T}^2$ where $A\in \RR[x,y,z]$. By \Cref{prop:sing-pt}, the vector field $\chi=(Ay,-Ax,0)$ has isolated singularities on $\mathbb{T}^2$ if and only if the hypersurface $\{A=0\}$ intersects $\mathbb{T}^2$ at isolated points.


We consider the projections 
$$p_1:\{(x,y,z)\in \mathbb{T}^2: z>0\}\to M~\mbox{and}~p_2:\{(x,y,z)\in \mathbb{T}^2: z<0\}\to M$$
from the upper and lower open half of the torus, respectively, onto the open annulus ${M:=\{(x,y)\in \RR^2: a^2-1< x^2+y^2< a^2+1\}}$, defined by 
$$p_1(x,y,z)=(x,y)~\mbox{and}~p_2(x,y,z)=(x,y).$$
Observe that $p_1$ and $p_2$ are diffeomorphisms with $$p_1^{-1}(x,y)=(x,y,\sqrt{1-(x^2+y^2-a^2)^2})~\mbox{and}~p_2^{-1}(x,y)=(x,y,-\sqrt{1-(x^2+y^2-a^2)^2}).$$ We denote the push-forwards of a vector field $\chi$ on $\mathbb{T}^2$ as $p_1(\chi)$ and $p_2(\chi)$ under the diffeomorphisms $p_1$ and $p_2$, respectively.

Suppose that $\mathcal{X}=(P,Q)$ is a planar vector field with an isolated singular point $p\in \RR^2$. We denote the jacobian matrix of the vector field $\mathcal{X}$ at $p$  as $J_{p}(\mathcal{X})=\begin{bmatrix}
    \frac{\partial P}{\partial x}(p) & \frac{\partial P}{\partial y}(p)\\
    \frac{\partial Q}{\partial x}(p) & \frac{\partial Q}{\partial y}(p) 
\end{bmatrix}$. We recall the following from \cite{DuLlAr06}.
\begin{definition}
    The singular point $p$ is called
    \begin{enumerate}
        \item semi-hyperbolic if $J_p(\mathcal{X})$ has exactly one non-zero eigenvalue,
        \item nilpotent if $J_p(\mathcal{X})$ is a non-zero matrix but both of its eigenvalues are zero,
        \item linearly zero if $J_p(\mathcal{X})$ is the zero matrix.
    \end{enumerate}
\end{definition}
Note that a semi-hyperbolic singular point is characterized in \cite[Theorem 2.19]{DuLlAr06}, a nilpotent singular point is characterized in \cite[Theorem 3.5]{DuLlAr06}. The local phase portrait of a linearly zero singular point can be studied  using the blow up technique, see \cite{alvarez2011survey} and \cite[Chapter 3]{DuLlAr06}.

Next, we classify the isolated singular points of the planar vector fields $p_1(\chi)$ and $p_2(\chi)$, where $\chi=(Ay,-Ax,0)$. Notice that $p_1$ and $p_2$ are orthogonal projections to the $xy$ plane. So, if we have the the local structure of trajectories of $p_1(\chi)$ (resp. $p_2(\chi)$) around the singular point $p_1(q)$ (resp. $p_2(q)$) for some $q\in \mathbb{T}^2$ then we can lift the structure on $\mathbb{T}^2$ to understand the local structure of trajectories of $\chi$ on $\mathbb{T}^2$ around $q$ since $\mathbb{T}^2$ is a 2-dimensional smooth manifold.
\begin{proposition}\label{thm:ch-sing-pt}
    Suppose that $q=(x_0,y_0,z_0)$ with $z_0>0$ (resp. $z_0<0$) is an isolated singular point on $\mathbb{T}^2$ of the vector field $\chi=(Ay,-Ax,0)$, where $A\in \RR[x,y,z]$. Then, the following holds. 
    \begin{enumerate}[(i)]
        \item If $div(p_1(\chi))\neq 0$ $($resp. $div(p_2(\chi))\neq 0)$ at $p_1(q)$ $($resp. $p_2(q))$, then $p_1(q)$ $($resp. $p_2(q))$ is a semi-hyperbolic singular point of $p_1(\chi)$ $($resp. $p_2(\chi))$.
        \item  If $div(p_1(\chi))= 0$ $($resp. $div(p_2(\chi))=0)$ at $p_1(q)$ $($resp. $p_2(q))$, then $p_1(q)$ $($resp. $p_2(q))$ is either a nilpotent or a linearly zero singular point of $p_1(\chi)$ $($resp. $p_2(\chi))$.
    \end{enumerate}
\end{proposition}
\begin{proof}
    We prove the statement for $z_0>0$. The proof for $z_0<0$ similarly follows.

    Notice that $p_1(\chi)=(By,-Bx)$ where $B=A(x,y,\sqrt{1-(x^2+y^2-a^2)^2})$. The jacobian matrix for the vector field $p_1(\chi)$ at $(x,y)$ is
    $$J_{(x,y)}(p_1(\chi))=\begin{bmatrix}
    B_x y & B_y y+B\\
    -B_x x-B & -B_y x
\end{bmatrix}.$$
As $(x_0,y_0,z_0)$ is an isolated singular point of $\chi$ on $\mathbb{T}^2$, $(x_0,y_0)$ is an isolated singular point of $p_1(\chi)$. Moreover, $A(x_0,y_0,z_0)=0$ and hence $B(x_0,y_0)=0$. So,
$$J_{(x_0,y_0)}(p_1(\chi))=\begin{bmatrix}
    B_x(x_0,y_0) y_0 & B_y(x_0,y_0) y_0\\
    -B_x(x_0,y_0) x_0 & -B_y(x_0,y_0) x_0\end{bmatrix}.$$

We see that $\det(J_{(x_0,y_0)}(p_1(\chi)))=0$. If $div(p_1(\chi))\neq 0$ at $p_1(q)$ then $tr(J_{(x_0,y_0)}(p_1(\chi)))\neq 0$. So, $p_1(q)$ is a semi-hyperbolic  singular point of $p_1(\chi)$.

Also, if $div(p_1(\chi))= 0$ at $p_1(q)$ then $tr(J_{(x_0,y_0)}(p_1(\chi)))= 0$. So, in this case, $p_1(q)$ is either a nilpotent or a linearly zero singular point.
\end{proof}

\noindent {\bf Acknowledgments.} 
The author thanks the anonymous reviewers for several helpful comments and suggestions, which have improved the quality of the paper. The author is grateful to Soumen Sarkar for encouragement, fruitful discussions, and helpful comments after going through the initial draft of this paper. The author thanks Joji Benny for helpful discussions. The author is supported by the Prime Minister's Research Fellowship, Government of India.
\bibliographystyle{abbrv}
\bibliography{biblio.bib}
\end{document}